\documentclass[12pt]{article}
\usepackage{amsmath, amsthm, amscd, amsfonts, amssymb, graphicx, color}
\usepackage[bookmarksnumbered, colorlinks, plainpages]{hyperref}
\usepackage[utf8]{inputenc}
\newtheorem{corollary}{Corollary}
\newtheorem{theorem}{Theorem}
\newtheorem{proposition}{Proposition}
\newtheorem{lemma}{Lemma}
\newcommand{\be}{\begin{equation}}
\newcommand{\ee}{\end{equation}}
\newcommand{\bp}{\begin{proposition}}
\newcommand{\ep}{\end{proposition}}
\newcommand{\bt}{\begin{theorem}}
\newcommand{\et}{\end{theorem}}
\title{ON A PROJECTIVELY INVARIANT PSEUDO-DISTANCE IN FINSLER GEOMETRY}
\author{ Behroz BIDABAD\thanks{The corresponding author, bidabad@aut.ac.ir}, Maryam SEPASI}
\date{ }

\begin{document}
\maketitle
\vspace{-3in}
Int. Journal of Geometric Methods in Modern Physics, vol.12, 2015\\
\ \ \ doi; 10.1142/s0219887815500437
\vspace{3in}\\
\begin{abstract}
Here, a non-linear analysis method is applied rather than classical one to study projective changes of  Finsler metrics.  More intuitively,  a projectively invariant pseudo-distance is introduced and characterized with respect to the Ricci tensor and its covariant derivatives.
\end{abstract}
\textbf{Keywords;} pseudo-distance; Finsler metric;   projective change;  Shwarzian derivative.

\section{Introduction}
 In physics, a geodesic as a generalization of straight line represents the equation of motion which determines all  phenomena as well as  geometry of the space.
Two regular metrics  on a manifold are said to be pointwise projectively related if they have the same geodesics as the point sets. Much of the practical importance of two projectively related ambient spaces derives from the fact that they produce same physical events, see  \cite{AM}.  In projective geometry, the classical method is studying projectively invariant quantities and characteristics of a regular metric on a manifold and applying them to present new  projectively invariant quantities or characteristics. For instance, Weyl tensor is one of the most important projectively invariant quantities in Finsler spaces which makes the constant curvature characteristic of a Finsler space to be projectively invariant.

Recently, an endeavor   has been made by the present authors, to define a projectively invariant  symmetric pseudo-distance $d_M$ on a  Finsler space $(M,F)$, cf., \cite{BS,SB}.  Here,   a reasonably comprehensive account  of analysis based on the methods of Schwarzian derivative is used to find some conditions under which this pseudo-distance is a distance. More intuitively, let $\gamma:=x^i(t)$ be a geodesic on $(M,F)$. In general, the parameter ``$t$" of $\gamma$, does not remain invariant under projective changes. There is a  parameter ``$p$" which remains invariant under projective changes called  \emph{projective }parameter.
In Refs. \cite{B, T} the projective parameter is defined  for geodesics of general affine connections. In the present work, it is shown that in a Finsler space the parameter ``$p$"  is a  solution of the following ODE
\begin{equation}\label{propara}
\{p,s\}:= \frac{\frac{d^3p}{ds^3}}{\frac{dp}{ds}}-\frac{3}{2}\Big[\frac{\frac{d^2p}{ds^2}}{\frac{dp}{ds}}\Big]^2 =\frac{2}{n-1}Ric_{jk}\frac{d{x}^j}{ds}\frac{d{x}^k}{ds},
\end{equation}
where $\{p,s\}$ is  known  in the literature  as \emph{Schwarzian derivative} and  ``$s$" is the arc length parameter of $\gamma$.  The projective parameter is unique up to a linear fractional transformations, that is
 \begin{equation}\label{PropertyOfSchwarzian}
 \{ \frac{a p + b}{c p +d} , s \} = \{p , s\},
 \end{equation}
where, $ad-bc\neq 0$. When
the Ricci tensor is parallel with respect to any of Berwald, Chern or Cartan
connection, it is constant along the geodesics and we can easily solve the equation (\ref{propara}).
 With this objective in mind,  first a  Finslerian setting of the Schwarz' lemma is carried out  as follows;
\begin{theorem}\label{Th;1} 
  Let $(M,F)$ be a connected Finsler space for which the Ricci tensor satisfies
  \be \label{shwarz condition}
 Ric_{ij} \leq -c^2g_{ij},
  \ee
 as matrices, for a  positive constant $c$.
  Then we have
  \begin{equation}\label{e26}
  {\tilde{f}}^*(ds_M^2) \leq \frac{n-1}{4c^2} ds_{_I}^2,
  \end{equation}
    where, $ds_{_I}$ and $ds_{_M}=\sqrt {g_{ij}(x , dx)dx^i dx^j}$ are the first fundamental  forms of the Poincar\'{e} metric on $I$
    and  the Finsler metric $F$ on $M$ respectively, and  $\tilde{f}$ is the natural lift of an arbitrary projective map ${f}$ on $TM$.
   \end{theorem}
 Next, the  Showarz' lemma is used to prove the following theorem.
 \bt \label{Th;2}
  Let $(M,F)$ be a connected Finsler space for which the Ricci tensor satisfies
  \begin{equation*}
  (Ric)_{ij} \leq -c^2g_{ij},
  \end{equation*}
 as matrices, for a positive constant $c$. Then the pseudo-distance $d_M$, is a distance.
  \et
Finally, we will  characterize the pseudo-distance  with respect to the Ricci tensor and its covariant derivative. More intuitively,
 These Theorems generalize some results in Riemannian spaces of Kobayashi  to Finsler spaces, cf.,  \cite{K}.
    Here we use notations of \cite{BCS,A}. Without pretending to be  exhaustive we quote some more significant works in projective geometry,  \cite{ Sh, ChSh,  RR, NT,  ZR}.
 \section{Preliminaries }

A (globally defined) Finsler structure on a differential manifold $M$ is a function $F: TM\rightarrow [0 , \infty) $ with the  properties, i) Regularity: $F$ is $C^{\infty}$ on the entire slit tangent bundle $TM_0$,
ii) Positive homogeneity:  $F(x , \lambda y) = \lambda F(x , y)$ for all $\lambda > 0$,
iii) Strong convexity: The Hessian matrix $(g_{ij}) := ({[1/2F^2]}_{y^iy^j})$ is positive-definite at every point of $TM_0$. The pair $(M,F)$ is known as a Finsler space.
Here and every where in this work the differential  manifold $M$ is supposed to be connected.
Every Finsler structure $F$ induces a spray $\textbf{G}=y^i\frac{\partial}{\partial x^i}-G^i(x,y)\frac{\partial}{\partial y^i}$ on $TM$, where $G^i(x,y):=\frac{1}{2} g^{il}\{[F^2]_{x^k y^l}y^k-[F^2]_{x^l}\}$. $\textbf{G}$ is a globally defined vector field on $TM$.  Projection of a flow line of $\textbf{G}$ on $M$
is called a geodesic . Differential equation of a geodesic in local coordinate is given by $\frac{d^2x^i}{ds^2}+G^i(x(s),\frac{dx}{ds})=0$, where $s(t) = \int_{t_0}^{t} F(\gamma , \frac{d\gamma}{dr}) dr$ is the arc length parameter.  For $x_0$ , $x_1$ $\in M$, denote by $\Gamma (x_0 , x_1)$ the collection of all piecewise $C^{\infty}$ curves  $\gamma : [a , b] \rightarrow M$ with $\gamma (a) = x_0$ and $\gamma(b) = x_1$. Define a map $d_F : M \times M \rightarrow [0 , \infty)$ by
\be\label{distance}
d_F(x_0 , x_1) := \inf L(\alpha),  \quad \alpha \in \Gamma (x_0 , x_1).
\ee
It can be shown that $d_F$ satisfies the first two axioms of a metric space. Namely,
\begin{itemize}
\item[(1)] $d_F(x_0 , x_1) \geq 0$ , where equality holds if and only if $x_0 = x_1$,
\item[(2)] $d_F(x_0 , x_1) \leq d_F(x_0 , x_1) + d_F(x_1 , x_2)$.
\end{itemize}
We should remark that the distance function $d_F$ on a Finsler space does not have  the symmetry property.
If the Finsler structure F is absolutely homogeneous, that is $F(x,\lambda y)=\mid \lambda \mid F(x,y)$ for $\lambda \in \mathbb{R}$, then one also has
\begin{itemize}
\item[(3)] $d_F(x_0, x_1) = d_F(x_1,x_0)$.
\end{itemize}
For a  non null $y \in T_xM$, the Riemann curvature
$\textbf{R}_y : T_xM \rightarrow T_xM$ is defined by $\textbf{R}_y(u)=R^i_k u^k \frac{\partial}{\partial x^i}$, where $R^i_k(y):=\frac{\partial G^i}{\partial x^k}-1/2\frac{{\partial}^2G^i}{\partial y^k \partial x^j}y^j+G^j \frac{{\partial}^2G^i}{\partial y^k \partial y^j}-1/2\frac{\partial G^i}{\partial y^j}\frac{\partial G^j}{\partial y^k}$.
The \emph{Ricci Scalar} is defined by $Ric:= {R^i}_i$, cf. \cite{BCS}.
 In the present work, we use the definition of  \emph{Ricci tensor}
introduced  by Akbar-Zadeh,  as follows  $Ric_{ik} :=\frac{1}{2}(F^2Ric)_{y^iy^k}$. cf., \cite{A}.
Let ${G^i}_j :=1/2 \frac{\partial G^i}{\partial y^j}$, $l^i:=\frac{y^i}{F}$, and  ${\widehat{l}}:=l^i\frac{\delta}{\delta x^i}=l^i( \frac{\partial}{\partial x^i}-{G^k}_i \frac{\partial}{\partial y^k})$.
By homogeneity we have $Ric_{ik} {\ell}^i{\ell}^k = Ric$. Let $\bar{F}$ be another Finsler structure on M. If any geodesic of $(M,F)$  coincides with a geodesic of
$ (M,\bar{F})$ as set of points and vice versa, then the change  $F\rightarrow \bar{F}$ of the metric is called \emph{projective } and $F$ is said to be \emph{projective} to $\bar{F}$.
  A Finsler space $(M,F)$  is projective to another Finsler space $(M,\bar{F})$, if and only if there exists a  1-homogeneous scalar field $P(x,y)$ satisfying
   \begin{equation}\label{e4}
    \bar{G}^i(x , y)=G^i(x , y)+P(x,y)y^i.
    \end{equation}
 The scalar field $P(x , y)$ is called the \emph{projective factor} of the projective change under consideration.

 It can be easily shown
 \begin{equation}\label{e6}
2F^2 {R^i}_k=2 (G^i)_{x^k}-\frac{1}{2}(G^i)_{y^j}(G^j)_{y^k}-y^j(G^i)_{y^kx^j}+G^j(G^i)_{y^ky^j}.
  \end{equation}
See Ref. \cite[ P.71]{BCS}. From (\ref{e6}) one  obtains
\be\label{e9}
2F^2 Ric = 2 (G^i)_{x^i} -\frac{1}{2} (G^i)_{y^j}(G^j)_{y^i}- y^j(G^i)_{y^ix^j} + G^j(G^i)_{y^iy^j}.
\end{equation}
Under the projective change (\ref{e4}) we have
\begin{equation}\label{e10}
\bar{F}^2\bar{Ric} = F^2Ric + \frac{(n-1)}{2}(\frac{\partial P}{\partial x^i} y^i - \frac{\partial P}{\partial y^i} G^i + \frac{P^2}{2}).
\end{equation}
 Now we are in a position to define in the next section  the projective parameter of a geodesic on a Finsler space.
\section{Projective invariant parameter on  Finsler spaces}
\subsection{Projective parameter}
In \cite{B}, Berwald has defined the projective parameter  for geodesics of general affine connections as a parameter which is  projectively invariant. He introduced the notion of a \emph{general affine connection $\Gamma$} on an $n$-dimensional manifold $M$, as a geometric object  with components ${\Gamma}^i_{jk}(x , \dot{x})$, 1-homogeneous in $\dot{x}$.  These geometric objects transform by the local change of  coordinates
\be\label{coordinatesChange}
{\bar{x}}^i = {\tilde{x}}^i(x^1,...,x^n),
\ee
as ${\tilde{\Gamma}}^i_{jk} = ({\Gamma}^l_{mr} \frac{\partial x^m}{\partial{\tilde{x}}^j} \frac{\partial x^r}{{\tilde{x}}^k} + \frac{{\partial}^2 x^l}{\partial{\tilde{x}}^j \partial{\tilde{x}}^k}) \frac {\partial {\tilde{x}}^i}{\partial x^l}$,
wherever $\dot{x}$ are transformed like the components of a contravariant vector. These specifications are carefully spelled out for geodesics of Finsler metrics in the following natural sense.
First recall that for a non-constant $C^\infty$ real function $f$ on $\mathbb{R}$, and for $t \in \mathbb{R}$, the \emph{Schwarzian derivative}
$\{f , t\}: = \frac{\frac{d^3f}{dt^3}}{\frac{df}{dt}}-\frac{3}{2}\Big[\frac{\frac{d^2f}{dt^2}}{\frac{df}{dt}}\Big]^2$,
 is defined to be an operator which is invariant under all linear fractional transformations
 $t \rightarrow \frac{a t + b}{c t +d}$ where, $ad - bc \neq 0$. That is,
 \be\label{PropertyOfSchwarzian}
 \{ \frac{a f + b}{c f +d} , t \} = \{f , t\}.
 \ee
Let $g$ be a real function for which the composition $f\circ g$ is defined. Then,
\be\label{schwarzian}
\{f\circ g,t\}=\{f,g(t)\}(\frac{dg}{dt})^2+\{g,t\}.
\ee
Let $\gamma$  be a geodesic on the Finsler space $(M , F)$. We need  a parameter $\pi$ which remains invariant  under both the coordinates change (\ref{coordinatesChange}),  and   the projective change (\ref{e4}).
We define the \emph{projective normal parameter} $\pi$ of   $\gamma$  by
\begin{equation}\label{e11}
\{\pi , s\} = -4 A G^0(x , \frac{dx}{ds}),
\end{equation}
where $\{\pi , s\}$ is the Schwarzian derivative,
 $A\neq0$ is a constant and
$ G^0(x , \dot{x})$  is a  homogeneous function of second degree in $\dot{x}$.
We require that the parameter $\pi$ remains invariant under the coordinates change (\ref{coordinatesChange}),  and   the projective change (\ref{e4}). This gives to the quantity $G^0$ the following   transformation laws,
 \begin{equation}\label{e12}
 \tilde{G}^0(\tilde{x} , \dot{\tilde{x}}) = G^0(x , \dot{x}), \qquad \dot{\tilde{x}}^i= \frac{\partial \tilde{x}^i}{\partial x^k} \dot{x}^k.
 \end{equation}
By projective change (\ref{e4}), we have
 \begin{equation}\label{e13}
\tilde{G}^0 = G^0 - \frac{1}{4A}(\frac{\partial P}{\partial x^i} \dot{x}^i - \frac{\partial P}{\partial \dot{x}^i} G^i + \frac{P^2}{2}).
\end{equation}
According to  (\ref{e10}) and (\ref{e13}),  the scalar $R^*$ defined by
 \begin{equation}\label{e14}
R^* := F^2Ric + 2A (n-1)G^0,
\end{equation}
 is 2-homogeneous in $\dot{x}^i$ and remains invariant under the  projective change (\ref{e4}).
If we put $R^* = 0$ then
 \begin{equation}\label{e15}
G^0 = -\frac{1}{2A(n-1)}F^2 Ric.
\end{equation}
Plugging the value of $G^0$ into  (\ref{e11}), we obtain
\begin{equation}\label{e16}
\{\pi,s\}=\frac{2}{n-1}F^2Ric = \frac{2}{n-1} Ric_{jk}\frac{d{x}^j}{ds}\frac{d{x}^k}{ds},
\end{equation}
which is called the  \emph{preferred projective normal parameter} up to linear fractional transformations.
In the sequel we will simply refer to preferred projective normal parameter as \emph{ projective parameter}.

 Let $(M,F)$ be projectively related to $(M,\bar F)$ and  the curve $\bar{x}(\bar{s})$ be a geodesic with affine parameter  $\bar{s}$ on $(M,\bar F)$  representing the same geodesic as $x(s)$  of $(M,F)$, except for its parametrization. Then, one can easily check that the projective parameter $\bar{\pi}$ defined by $\bar{x}(\bar{s})$ is related to the projective parameter $\pi$  by  $\bar{\pi} = \frac{a\pi+b}{c\pi+d}$.
 \subsection{Projective parameter for Ricci parallel Finsler spaces}
\bp\label{Pr;1}
Let $(M,F)$ be a Finsler space of parallel Ricci tensor. Then the Ricci tensor is constant along the geodesics parameterized by arc-length, and solutions of (\ref{e16}) are given as follows.
\begin{itemize}
\item[i)] If $\{p,s\}=c^2$ with $c>0$ then
\begin{equation}\label{p1}
p=\frac{\alpha cos(cs)+\beta sin(cs)}{\gamma cos(cs)+\delta sin(cs)}.
\end{equation}
\item[ii)] If $\{p,s\}=-c^2$ with $c>0$ then
\begin{equation}\label{p2}
p=\frac{\alpha e^{cs}+\beta e^{-cs}}{\gamma e^{cs}+\delta e^{-cs}}.
\end{equation}
\item[iii)] If $\{p,s\}=0$ then
\begin{equation}\label{p3}
p=\frac{\alpha+\beta s}{\gamma+\delta s}.
 \end{equation}
 \end{itemize}
 \ep
 \begin{proof}
  Let   Ricci tensor be parallel with respect to the Cartan connection. We denote the horizontal and vertical Cartan covariant derivatives of Ricci tensor by
   $\bigtriangledown^c_{\frac{\delta}{\delta x^k}}Ric_{ij}$ and $\bigtriangledown^c_\frac{\partial}{\partial y^k}Ric_{ij}$ respectively. We have
 \begin{equation}\label{c}
\bigtriangledown^c_{\frac{\delta}{\delta x^k}}Ric_{ij}=\frac{\delta Ric_{ij}}{\delta x^k}- Ric_{ir}{\Gamma^r}_{jk}-Ric_{jr}{\Gamma^r}_{ik}=0,
\ee
\begin{equation}\label{d}
\bigtriangledown^c_{\frac{\partial}{\partial y^k}}Ric_{ij}=\frac{\partial Ric_{ij}}{\partial y^k}- Ric_{ir}\frac{{A^r}_{jk}}{F}-Ric_{jr}\frac{{A^r}_{ik}}{F}=0,
\ee
where ${\Gamma^i}_{jk}=\frac{1}{2}g^{ih}(\frac{\delta g_{hj}}{\delta x^k}+\frac{\delta g_{kh}}{\delta x^j}-\frac{\delta g_{jk}}{\delta x^h})$  and ${A^i}_{jk}:=g^{ih}A_{hjk}=g^{ih}\frac{F}{4}\frac{\partial{g_{ij}}}{\partial y^k}$ are the components of  \emph{Cartan} tensor.
Consider  the geodesic $\gamma:=x^i(s)$, where $s$ is the arc-length parameter. Contracting (\ref{c}) by $\frac{dx^i}{ds}\frac{dx^j}{ds}\frac{dx^k}{ds}$ gives $$\frac{dx^i}{ds}\frac{dx^j}{ds}\frac{dx^k}{ds}\big [(\frac{\partial Ric_{ij}}{\partial x^k}-{G^l}_k\frac{\partial Ric_{ij}}{\partial y^l})-
(Ric_{ir}{\Gamma^r}_{jk})-(Ric_{jr}{\Gamma^r}_{ik}) \big ]=0.$$
Using (\ref{d}), and the property $y^j {A^i}_{jk}=y^k {A^i}_{jk}=0$ of Cartan tensor, we have
$$\frac{dx^i}{ds}\frac{dx^j}{ds}\big [\frac{dRic_{ij}}{ds}-\frac{dx^k}{ds}{G^l}_k(Ric_{ir}\frac{{A^r}_{jl}}{F}+Ric_{jr}\frac{{A^r}_{il}}{F})-2\frac{dx^k}{ds}Ric_{jr}{\Gamma^r}_{ik} \big] =0.$$
Therefore
$$\frac{dRic_{ij}\frac{dx^i}{ds}\frac{dx^j}{ds}}{ds}-2Ric_{ij}\frac{d^2x^i}{ds}\frac{dx^j}{ds}-0+2Ric_{rj}\frac{d^2x^r}{ds}\frac{dx^j}{ds}=0,$$
which leads to
\begin{equation}\label{e}
Ric_{ij}\frac{dx^i}{ds}\frac{dx^j}{ds}=constant.
\ee
Following the  method  just used, we can prove that if the Ricci tensor is parallel with respect to the Berwald or Chern connection then along  the geodesic $\gamma$ parameterized by arc-length,  we have $Ric_{ij}\frac{dx^i}{ds}\frac{dx^j}{ds}=constant$. By some direct calculations, one can show  the general solution of (\ref{e16}) is given by
\begin{equation}\label{b}
u(t)=\frac{\alpha y_1 + \beta y_2}{\gamma y_1 + \delta y_2},
\ee
where $\alpha \delta-\beta \gamma\neq 0$ and  $y_1$ and $y_2$ are linearly independent solutions of  the ordinary differential equation
\begin{equation}\label{a}
y^{''}+Q(s)y(s)=0,
\end{equation}
where $Q(s)=\frac{1}{n-1}Ric_{ij}\frac{dx^i}{ds}\frac{dx^j}{ds}.$
 Thus the equation (\ref{e16})  reduces to a second order ODE with constant coefficient and with respect to the sign of Ricci tensor, one can explicitly determine   a projective parameter $p$ as an elementary function of $s$ by (\ref{p1}), (\ref{p2}) and (\ref{p3}). This completes the proof.
 \end{proof}
\section{Intrinsic pseudo-distance }
Consider the open interval  $I=(-1,1)$ with Poincar\'e metric $ds^2_I=\frac{4du^2}{(1-u^2)^2}$. The Poincar\'e distance between two points $a$ and $b$ in $I$ is given by
\begin{equation}\label{e12}
\rho(a , b) = \mid \ln \frac{(1-a)(1+b)}{(1-b)(1+a)} \mid,
\end{equation}
cf.,  \cite{O}.
 A geodesic $f :I \rightarrow M$ on the Finsler space $(M,F)$ is said to be \emph{projective}, if the natural parameter $u$ on $I$ is a projective parameter.
We now come to the main step for definition of the pseudo-distance $d_M$, on $(M, F)$. To do so, we proceed in analogy with the treatment of Kobayashi in Riemannian geometry, cf., \cite{K}. Although he has confirmed that the construction of intrinsic pseudo-distance is valid for any manifold with an affine connection, or more generally a projective connection, cf., \cite{K1}, we restrict our consideration to  the pseudo-distances induced by the Finsler structure $F$ on a  connected manifold $M$.
Given any two points $x$ and $y$ in $(M,F)$, we consider a chain $\alpha$ of geodesic segments joining these points. That is
\begin{itemize}
\item a chain of points $x = x_0 , x_1 , ... ,x_k = y$ on $M$;
\item pairs of points $a_1,b_1 ,..., a_k,b_k$ in $I$;
\item projective maps $f_1,...,f_k$, $f_i: I \rightarrow M $, such that
$$f_i(a_i) = x_{i-1}, \quad f_i(b_i) = x_i, \quad i = 1,...,k.$$
\end{itemize}
By virtue of the Poincar\'{e} distance $\rho(.,.)$ on $I$ we define the length $L(\alpha)$ of the chain $\alpha$ by
$L(\alpha):= \Sigma_i \rho(a_i , b_i)$, and we put
\begin{equation}\label{e19}
d_M(x , y):= inf L(\alpha),
\end{equation}
where the infimum is taken over all chains $\alpha$ of geodesic segments from $x$ to $y$.
\bp
Let $(M, F)$ be a Finsler space. Then for any points $x$, $y$, and $z$ in $M$,  $d_M$ satisfies
 \begin{itemize}
\item[(i)] $d_M(x,y)=d_M(y,x)$.
\item [(ii)]$d_M(x , z) \leq d_M(x , y) + d_M(y , z)$.
\item [(iii)]If $x = y$ then $d_M(x , y) = 0$ but the inverse is not  always true.
 \end{itemize}
\ep
\begin{proof}
 To prove $(i)$, consider the chain $\bar \alpha$ of geodesic segments consisting of; (1) a chain of points $y=\bar x_0 , \bar x_1 , ... , \bar x_k = x$ on $M$; (2) pairs of points $\bar a_1,\bar b_1 ,...,\bar a_k, \bar b_k$ in $I$, where $\bar a_i:= b_{k+1-i}$ and  $\bar b_i:= a_{k+1-i} \quad i = 1,...,k$ ; (3) projective maps $\bar f_1,...,\bar f_k$, $\bar f_i: I \rightarrow M $, where $\bar f_i:=f_{k+1-i} \quad i = 1,...,k$.  We have
$\bar f_i(\bar a_i) = f_{k+1-i}(b_{k+1-i})= x_{k+i-1}=\bar x_{i-1}$ and $\bar f_i(\bar b_i) = f_{k+1-i}(a_{k+1-i})= x_{k-i}=\bar x_i$. The length $L(\bar \alpha)$ of the chain $\bar \alpha$ by definition is $L(\bar \alpha) = \Sigma_i  \rho(\bar a_i ,\bar  b_i)= \Sigma_i  \rho(b_i , a_i)=L(\alpha)$. Therefore we get an elegant result reads  $d_M$ is symmetric, i.e. $d_M(x,y)=d_M(y,x)$.\\
 To prove $(ii)$,  it is enough to  show  for all positive $ \epsilon > 0$, the inequality  $ d_M(x , z) \leq d_M(x , y) + d_M(y , z) + \epsilon$, holds.
 There is a chain $\alpha_1$ joining the points $x$ and $y$ through the projective maps $f_i$, for $i=1,...,k_1$ and a chain $\alpha_2$ joining $y$ and $z$ through the projective maps $g_j$, for $j=1,...,k_2$ such that
  $$d_M(x , y) \leq L(\alpha_1) \leq d_M(x , y) +\epsilon /2,$$
  $$d_M(y , z) \leq L(\alpha_2) \leq d_M(y , z) +\epsilon /2.$$
  We define the chain $\alpha$ joining $x$ and $z$ through the projective maps $h_k$, for $k=1,...,k_1 + k_2$ such that
  $$h_k = f_k, \ \ \ k=1,...,k_1, \qquad \qquad $$
   $$ \quad h_k = g_{k - k_1}, \ \ \ k=k_1+1,...,k_1 + k_2.$$
   From which we conclude
   $$d_M(x , z) \leq L(\alpha) \leq L(\alpha_1) + L(\alpha_2) \leq d_M(x , y) + d_M(y , z) + \epsilon .$$
To show the assertion $(iii)$, we consider the Euclidian space $\mathbb{R}^2$. Let $A=(-1/2, 0)$ and $B=(0,1/2)$ be two points in $\mathbb{R}^2$. Here, we have $Ric(x,y)=0$, $x\in \mathbb{R}^2$ and $y\in T_x\mathbb{R}^2$. Thus $\pi(s)=s$ is a special solution of the ODE $\{\pi,s\}=0$. For $n\in \mathbb{N}$, let $\alpha_n$ be the chain of geodesic segments joining the points $A$ and $B$, with; 1) pairs of points $a_n=-1/2n$ and $b_n=1/2n$ in $I$, 2) projective map $f_n:I\rightarrow \mathbb{R}^2$, $f_n(\pi)=(n\pi,0)$. We have $f_n(-1/2n)=(-1/2,0)$ and $f_n(1/2n)=(1/2,0)$. Moreover $\rho(-1/2n,1/2n)=\mid\ln\frac{(1+1/2n)^2}{(1-1/2n)^2}\mid$. Considering $n$ sufficiently large, we have $d_M(A,B)=inf(L(\alpha))=0$. This completes the proof.
   \end{proof}
We call $d_M(x,y)$ the \emph{ pseudo-distance} of any two points $x$ and $y$ on $M$.
By means of  the property (\ref{PropertyOfSchwarzian}) of Schwarzian derivative and the fact that the projective parameter is invariant under fractional transformations, the pseudo-distance $d_M$ is projectively invariant.

 \bp\label{Prop;3}
Let $(M,F)$ be a Finsler space.
\begin{itemize}

 \item [(a)] If the geodesic  $f : I \rightarrow M $ is projective, then
 \begin{equation}\label{e23}
\rho(a , b) \geq d_M (f(a) , f(b)), \ \ \ \ a,b \in I.
 \end{equation}
 \item [(b)] If $\delta_M$ is any pseudo-distance on $M$ with the property
  $$\rho(a , b) \geq \delta_M(f(a) , f(b)), \ \ \ \ a,b \in I,$$
   and for all projective maps $f: I \rightarrow M$, then
  \begin{equation}\label{e24}
 \delta_M(x , y)\leq d_M(x , y), \ \ \ \ \ \  x,y \in M.
  \end{equation}
  \end{itemize}
  \begin{proof}
(a) By definition $d_M$ is supposed to be the infimum of $L(\alpha)$ for all chain $\alpha$  and actually $f$ is one of them.\\
(b) For $x,y \in M$ consider an arbitrary  chain of projective segments $\alpha$, satisfying
  $x = x_0,...,x_k = y$,   $a_1,b_1,...,a_k,b_k \in I$, and projective maps $f_1,...,f_k$, $f_i:I \rightarrow M$, such that
  $$f_i(a_i) = x_{i-1} \ \ \ , \ \ \ f_i(b_i) = x_i.$$
   We have by assumption
  $$L(\alpha) = \Sigma \rho(a_i , b_i) \geq \Sigma \delta_{_M} (f(a_i) , f(b_i)).$$
  So for an arbitrary chain $\alpha$, the triangle inequality  for the pseudo-distance $\delta_M$ leads to
  $$L(\alpha) \geq \delta_{_M} (f(a_1) , f(b_k))=\delta_{_M}(x,y).$$
   Therefor $\delta_M(x , y)$ is a lower bound for $L(\alpha)$ and $inf_{\alpha} \ \ L(\alpha) \geq \delta_M(x , y)$. Finally we have $d_{_M}(x,y )\geq \delta_{_M}(x , y)$. This completes the proof.
  \end{proof}
  \ep
  \subsection{Proof of the Schwarz' lemma on Finsler Spaces}
  Let $ds_{_I}=2\frac{ dx }{1-x^2}$ be the first fundamental  form related to the Poincar\'{e} metric  on the open interval $I$, and $ds_{_M}=\sqrt {g_{ij}(x , dx)dx^i dx^j}$. Denote by  $\tilde{f}$  the natural lift of a projective map ${f}$ to the tangent bundle $TM$.
   Now we are in a position to prove the Schwarz' lemma in Finsler geometry.

{ \it Proof of Theorem \ref{Th;1}.}
    Let $f : I\rightarrow M$ be an arbitrary projective map.
  We denote the projective  and  arc-length parameters of $f$ by ``$u$"  and  ``$s$", respectively.
   Let us put $h = \frac{\tilde{f}^*(ds_{_M})}{ds_{_I}}$.
  In order to find an upper bound for $h$ in the open interval $I$, first we assume that $h$ attains its maximum in the interior of $I$. By means of $ds_{_I}=2\frac{ du }{1-u^2}$, we have
     $h = \frac{1}{2} (1-u^2)\frac{ds}{du}.$
     Thus     \begin{equation}\label{Eq; ln h}
   \frac{d \ln h}{du}= \frac{-2u}{1-u^2}+\frac{s^{''}}{s^{'}}.
    \end{equation}
    At the maximum point of $h$, $\frac{d \ln h}{du}$ vanishes and we have
     \begin{equation}\label{e29}
   \frac{s^{''}}{s^{'}} =\frac{2u}{1-u^2}.
       \end{equation}
       The second derivative of (\ref{Eq; ln h}) yields
       $$\frac{d^2\ln h}{du^2} = \frac{s^{'''}s^{'} - (s^{''})^2}{(s^{'})^2} + \frac{-2(1 - u^2) - 4u^2}{(1 - u^2)^2} $$
       $$=\frac{s^{'''}}{s^{'}} - (\frac{s^{''}}{s^{'}})^2 - 2\frac{1 + u^2}{(1 - u^2)^2 } $$
       $$= \{s , u\} + \frac{1}{2}(\frac{s^{''}}{s^{'}})^2 - 2\frac{1 + u^2}{(1- u^2)^2}.$$
       By virtue of  (\ref{schwarzian}), the parameters $u$ and $s$ satisfy
       $\{u , s\} = - \{s , u\} (\frac{du}{ds})^2.$
       Applying (\ref{e29}), we have
       \be\label{dln}
        \frac{d^2\ln h}{du^2}= \frac{-2}{(1- u^2)^2}- \{u,s\} (\frac{ds}{du})^2.
        \ee
       At the maximum point of $h$, the second derivative should be negative, $\frac{d^2\ln h}{du^2} \leq 0$.  On the other hand, by means of
          $\{u , s\} = \frac{2}{n-1}(Ric)_{ij} \frac{dx^i}{ds} \frac{dx^j}{ds},$
          and  the assumption (\ref{shwarz condition}), we get
           $\{u , s\} \leq \frac{-2c^2}{n-1}g_{ij} \frac{dx^i}{ds} \frac{dx^j}{ds}.$
           For the arc-length parameter $s$, we have $g_{ij} \frac{dx^i}{ds} \frac{dx^j}{ds} = 1$.  Therefore
            \begin{equation}\label{e31}
          \{u , s\} \leq \frac{-2c^2}{n-1}< 0.
             \end{equation}
           Considering the above property, (\ref{dln}) reads
        \begin{align*}
        & \frac{-2}{(1- u^2)^2}- \{u,s\} (\frac{ds}{du})^2\leq 0,\\
         & (1-u^2)^2(\frac{ds}{du})^2\leq-\frac{2}{\{u,s\}}.
         \end{align*}
         By definitions of $h$ and $\{u , s\}$, the left hand side  and the right hand side are equal to $4h^2$ and $-\frac{(n-1)}{Ric_{ij}\frac{dx^i}{ds}\frac{dx^j}{ds}}$, respectively.  By  the assumption (\ref{shwarz condition}), the right hand side is bounded above by $(n-1)/c^2$. As claimed, this  happens at the maximum point of $h$. Therefore, for any $u$ in $I$, (\ref{e26}) holds well. This completes the proof when $h$ attains
its maximum in $I$. In general, we consider  a positive number $r<1$ and $I_r=\{-r<u<r\}$ with Poincar\'{e} metric $4 r^2\frac{du^2}{(r^2-u^2)^2}$. The function $h_r=\frac{\tilde{f}^* ds_{_M}}{{ds_{_I}}_r}$ vanishes on the boundary of the interval $I_r$. Therefore it takes its maximum in the interior of $I_r$. Applying the above discussion  to $h_r$, we have $4h^2_r\leq (n-1)/c^2$. To complete the proof, we let $r\rightarrow 1$.
 \hspace{\stretch{1}}$\Box$\\
 \begin{corollary}
          Let $(M,F)$ be a Finsler space for which the Ricci tensor satisfies
  \begin{equation*}
   (Ric)_{ij} \leq -c^2g_{ij},
  \end{equation*}
  as matrices, for a positive constant $c$. Let $d_F(. , .)$ be the distance induced by $F$, then for every projective map $f:I \rightarrow M$, $d_F$ is bounded above by the Poincar\'e distance $\rho$, that is
   \begin{equation}\label{e32}
\rho(a , b) \geq \frac{2 c}{\sqrt{n - 1} } d_F(f(a) , f(b)).
 \end{equation}
 \end{corollary}
\begin{proof}
By means of Theorem \ref{Th;1} we have
 $$(ds)^2 \leq \frac{ (n - 1)}{4c^2}ds_I^2,$$
 that is $\frac{2 c}{\sqrt{n - 1} } ds \leq ds_{_I}.$
 By integration we obtain (\ref{e32}), what is claimed.
\end{proof}

Now we are in a position to prove that the pseudo-distance $d_M$ is a distance on $M$.

 \textit{Proof of Theorem \ref{Th;2}.}
To establish the proof we have only to show that if  $d_M(x , y) = 0$ then $x = y$.
  By Proposition \ref{Prop;3} and the above corollary we get
  $$ d_F(x , y) \frac{2 c}{\sqrt{n - 1}} \leq d_M(x , y).$$
If  $d_M(x , y) = 0$ then $d_F(x , y) = 0$ and $x = y$. Thus the pseudo-distance $d_M$ is a distance. This completes the proof.
 \hspace{\stretch{1}}$\Box$\\

 As a conclusion we have the following two Theorems.
\bt\label{Th;3}
Let $(M,F)$ be a connected complete Finsler space of positive semi-definite  Ricci tensor. Then the intrinsic projectively invariant pseudo-distance is trivial, that is $d_M=0$.
\et
This result is a generalization of a theorem in Riemannian geometry, cf. \cite{KS}.
To prove  Theorem \ref{Th;3}, we need the following  two Lemmas.
\begin{lemma}\label{Lem;1}
Let $(M,F)$ be a complete Finsler space and $x_0$ and $x_1$  two points on $M$. If there is a geodesic $x(u)$ with projective parameter $u$, $-\infty <u <+ \infty$, such that $x_0=x(u_0)$ and $x_1=x(u_1)$ for some $u_0$ and $u_1$ in $\mathbb{R}$ then
$d_M(x_0,x_1)=0.$
\end{lemma}
\begin{lemma}\label{Lem;2}
Let $(M,F)$ be a complete Finsler space and $x(s)$  a geodesic with arc-length parameter $-\infty<s<\infty$. Assume that there exists a (finite or infinite) sequence of open intervals $I_i=(a_i,b_i)$, $i=0,\pm 1,\pm 2,...$, such that;\\ i) $a_{i+1}\leq b_i$, $lim_{i\rightarrow -\infty} a_i=-\infty$ and $lim_{i\rightarrow \infty}b_i=+\infty$  Such that $\bigcup_i \overline{I_i}=(-\infty,+\infty)$;\\ ii) in each interval $I_i=(a_i,b_i)$, a projective parameter ``$u$" moves from $-\infty$ to $+\infty$ whenever $t$ moves from $a_i$ to $b_i$. Then, for any pair of points $x_0$ and $x_1$ on this geodesic, we have
$$d_M(x_0,x_1)=0.$$
\end{lemma}

Proof of these two Lemmas is an straight forward application of Schwarz's Lemma and will appear in details in \cite{BS}.
 It is well known for $Q(s)\geq 0$ in equation (\ref{a}),  the open intervals $I_i$ in Lemma \ref{Lem;2} can be constructed through some elementary lemmas in  \cite{BS} where we don't  mention the lemmas here to avoid  repetition.  Construction of such open intervals $I_i$ completes the proof of  Theorem \ref{Th;3}.
\subsection{Parallel negative-definite Ricci tensor}

\bt\label{Th;4}
Let $(M,F)$ be a connected (complete) Finsler space of negative-definite Ricci tensor  parallel with respect to  Berwald or Chern connection. Then  the intrinsic projectively invariant pseudo-distance $d_M$, is a (complete) distance.
\et

 Proof of Theorem \ref{Th;4} is based on the Finsler structure defined by the negative Ricci tensor $Ric_{ij}$ as follows. Let $\hat{F}(x,y)=\sqrt{-Ric_{ij}(x,y)y^iy^j}$, it will be shown in  \cite{BS}  that  $\hat{F}$ is a Finsler structure on  $M$.
It can be shown  that the spray coefficients of  $\hat{F}$ and $F$ are equal, that is ${\hat{G}}^i=G^i$. It shows, according to   Theorem \ref{Th;2},  $d_M$ is a (complete) distance.
 Therefor  $Ric_{ij}$ is the Ricci tensor of $\hat{F}$ too. According to Theorem \ref{Th;2}, $d_{M}$ is a (complete) distance.

 Faculty of Mathematics and Computer Science,\\ Amirkabir University of Technology, \\ bidabad@aut.ac.ir\\

 Faculty of  Science, Islamic Azad University of Shiraz,\\ Sadra, Shiraz, 71993, Iran. \\sepasi@iaushiraz.ac.ir

\end{document}